\documentclass[10pt, letterpaper]{article}

\usepackage[pdftex]{graphicx,color}  
\usepackage{amsfonts, amssymb, amscd, amsmath, amsthm, fullpage, gensymb, hyperref}
\usepackage[lite]{amsrefs}
\usepackage{enumerate}
\usepackage{color, colortbl}

\definecolor{Gray}{gray}{0.85}

\newtheorem{theorem}{Theorem}[section]

\newtheorem{lemma}[theorem]{Lemma}

\theoremstyle{definition}
\newtheorem{definition}{Definition}[section]
\newcommand{\Hys}{\mathbb{H}^3}

\title{Hyperbolic triangular prisms with one ideal vertex}

\author{Grant S. Lakeland and Corinne G. Roth}

\begin{document}

\maketitle

\begin{abstract}In this paper, we classify all of the five-sided three-dimensional hyperbolic polyhedra with one ideal vertex, which have the shape of a triangular prism, and which give rise to a discrete reflection group. We show how to find each such polyhedron in the upper half-space model by considering lines and circles in the plane. Finally, we give matrix generators in $\mathrm{PSL}_2(\mathbb C)$ for the orientation-preserving subgroup of each corresponding reflection group.\end{abstract}

\section{Introduction}\label{section:intro}

A convex polyhedron in hyperbolic $3$-space $\Hys$ generates a discrete group of isometries if the dihedral angles at which its bounding planes meet are all integer submultiples of $\pi$ -- that is, each angle is of the form $\pi/n$ radians, for an integer $n \geq 2$ -- and if the dihedral angles satisfy some other combinatorial criteria. The set of all such polyhedra is infinite, with some partial classifications completed. For example, the fewest sides such a polyhedron may have is four, and the 32 hyperbolic tetrahedra were found by Lann\'{e}r \cite{lanner}, Vinberg \cite{vinberg}, and Thurston \cite{Thurston-notes}. The 825 smallest volume all-right-angled polyhedra have been found by Inoue \cite{inoue}.

For certain computations, it is helpful to know matrix generators in $\mathrm{PSL}_2(\mathbb C)$ for the orientation-preserving index 2 subgroups of these reflection groups. Such generators were found for some of the tetrahedral groups by Brunner, Lee, and Wielenberg \cite{brunneretal}, and the second author has shown how to find these for all 32 tetrahedra \cite{Lakeland}. These matrix generators have recently been used, for example, by Hoffman \cite{Hoffman} to study knot complements which cover tetrahedral orbifolds, and by \c Seng\"un \cite{Sengun} to study growth in torsion homology of subgroups of the tetrahedral groups.

In this paper, we perform the same calculations for one class of five-sided polyhedra, which when drawn schematically resemble triangular prisms. The set of all such polyhedra is infinite, and for simplicity we restrict our attention to such polyhedra which have one ideal vertex and five finite vertices. As such, all of our examples are non-compact. We find:

\begin{theorem}Up to relabeling by symmetry of the prism, there are 12 infinite families of prisms with exactly one ideal vertex, each of which has one dihedral angle which may be any integer submultiple of $\pi$ less than or equal to $\pi/6$. In addition, there are 78 specific arrangements, where in each case all dihedral angles not incident to the ideal vertex are at least $\pi/5$.  \end{theorem}

We note that the list of polyhedra considered in this paper overlaps with those considered by Kaplinskaja \cite{Kaplinskaja}, who studied finite volume simplicial prisms in $\Hys$, $\mathbb{H}^4$ and $\mathbb{H}^5$ which give rise to discrete reflection groups, and listed their Coxeter graphs. Each of our prisms either appears there, or can be decomposed into two polyhedra which do. As such, our classification is not new, although we do not believe that our polyhedra have previously been listed in this way. The Coxeter graphs of these polyhedra do not immediately allow one to produce isometries which generate the orientation-preserving subgroup, and in this paper we provide a method for this.

The method used is the following. First, we use Andreev's Theorem to set up the combinatorial rules which the dihedral angles must satisfy, and we find all admissible arrangements of angles which satisfy these rules. Then, we reduce the problem of finding hyperbolic planes in the upper half-space $\Hys$ which meet at these prescribed angles to a similar problem, involving finding lines and circles in the plane which meet at the same angles. Finally, we use this geometric data to write down matrix generators for each group.

This paper is organized as follows. After some geometric preliminaries in Section 2, in Section 3 we describe all of the possible arrangements of angles which are possible for our prisms, grouped by the possible angles at the ideal vertex. We also outline a method to locate each prism precisely in $\Hys$. In Section 4 we describe a general method to find each prism and write down corresponding matrices which generate the orientation-preserving subgroup of isometries of $\Hys$, and in Section 5 we summarize the possible angle arrangements in tables.

{\bf Acknowledgments.} We wish to thank Neil Hoffman for suggesting this project, and for help checking the results, and Charles Delman for his guidance and support. We thank Matthieu Jacquemet and the referee for helpful comments.

\section{Geometric preliminaries}\label{section:geom}

In this section, we recall some definitions and results about hyperbolic polyhedra. We will work in the upper half-space model for $\Hys$, $\{ (x,y,z) \in \mathbb{R}^3 \mid z>0 \}$, and we recall that in this model, geodesic lines are vertical lines and semicircles which meet the plane $\{z=0\}$ perpendicularly, and geodesic planes are vertical planes and hemispheres whose equators lie in the plane $\{ z=0 \}$. 

We first note that in order for a polyhedron to generate a discrete reflection group, all of its dihedral angles must be integer submultiples of $\pi$ radians, and the integer must be no less than $2$. In this paper, we will label a dihedral angle of $\pi/n$ by the natural number $n$.

\begin{definition}\label{def:euclidean}A triangle is \emph{Euclidean} if its angles $p, q,$ and $r$ satisfy the equation
\[ \frac{\pi}{p} + \frac{\pi}{q} + \frac{\pi}{r} = \pi\]
or
\[ \frac{1}{p} + \frac{1}{q} + \frac{1}{r} = 1.\] \end{definition}

\begin{definition}\label{def:spherical}A triangle is \emph{spherical} if its angles $p, q,$ and $r$ satisfy the inequality
\[ \frac{1}{p} + \frac{1}{q} + \frac{1}{r} > 1.\] \end{definition}

\begin{definition}\label{def:hyperbolic}A triangle is \emph{hyperbolic} if its angles $p, q,$ and $r$ satisfy the inequality
\[ \frac{1}{p} + \frac{1}{q} + \frac{1}{r} < 1.\] \end{definition}

With these definitions in mind, we will appeal to Andreev's Theorem \cite{Andreev} for hyperbolic polyhedra, which specifies exactly what conditions a combinatorial arrangement of dihedral angles must satisfy in order that it give rise to a hyperbolic polyhedron. For the precise statement given below, we refer to Dunbar, Hubbard, and Roeder \cite{Roeder}.

\begin{theorem}[Andreev]\label{andreev} If P is  a  compact,  finite-sided  hyperbolic  polyhedron  with  dihedral angle  $\alpha_i$ at each edge $e_i$, then the following conditions hold:

\begin{enumerate}\item For each $i$, $\alpha_i > 0$;
\item If three edges $e_i, e_j, e_k$ meet at at a vertex, then $\alpha_i + \alpha_j + \alpha_k > \pi$;
\item If there exists a prismatic $3$-circuit intersecting $e_i$, $e_j$ and $e_k$, then $\alpha_i + \alpha_j + \alpha_k < \pi$;
\item If there exists a prismatic $4$-circuit intersecting $e_i$, $e_j$, $e_k$ and $e_l$, then $\alpha_i + \alpha_j + \alpha_k + \alpha_l < 2\pi$; and
\item For a quadrilateral face with edges enumerated successively $e_1$, $e_2$, $e_3$ and $e_4$, and $e_{12}$, $e_{23}$, $e_{34}$, and $e_{41}$ are such that $e_{12}$ is the third edge meeting at the vertex where $e_1$ and $e_2$ intersect (and similarly for other $e_{ij}$, then
\begin{enumerate}[(a)]\item $\alpha_1 + \alpha_3 + \alpha_{12} + \alpha_{23} + \alpha_{34} + \alpha_{41} < 3\pi$; and
\item $\alpha_2 + \alpha_4 + \alpha_{12} + \alpha_{23} + \alpha_{34} + \alpha_{41} < 3\pi$.
\end{enumerate}
\end{enumerate}
\end{theorem}


The dihedral angle of intersection of two planes in the upper half-space model of $\Hys$ is the same as the angle between the respective tangent planes at any point of intersection. Since these planes are either vertical Euclidean planes or Euclidean spheres with center on the plane $\{ z=0 \}$, if two planes intersect, they have a common point on the plane $\{ z=0 \}$. In this case, the respective tangent planes are both vertical Euclidean planes, and so the dihedral angle is the angle the tangent planes make in the $x$-$y$ plane. 

With this in mind, we observe that the aim of finding five hyperbolic planes which intersect at prescribed angles may be reduced to finding five lines and circles in the $x$-$y$ plane which intersect at the same prescribed angles. Since three of our planes intersect at an ideal vertex, we may place this ideal vertex at $\infty$, and thereby assume that the three planes are vertical Euclidean planes. These will correspond to Euclidean lines in the $x$-$y$ plane. The remaining two sides will then correspond to circles in the $x$-$y$ plane; since angles of intersection are preserved by Euclidean similarities, we may take one of the circles to be the unit circle in the $x$-$y$ plane.

With these assumptions about our setup in place, when finding lines and circles which intersect at given angles, we will appeal frequently to the following results.

\begin{lemma}\label{lem:thetaline}  A line that intersects a circle of radius $r$ at angle $\theta$ comes distance \textit{$r\cos(\theta)$} away from the center of the circle at its closest point.\end{lemma}

\begin{figure}[htb]
\begin{center}
\includegraphics[scale=0.45]{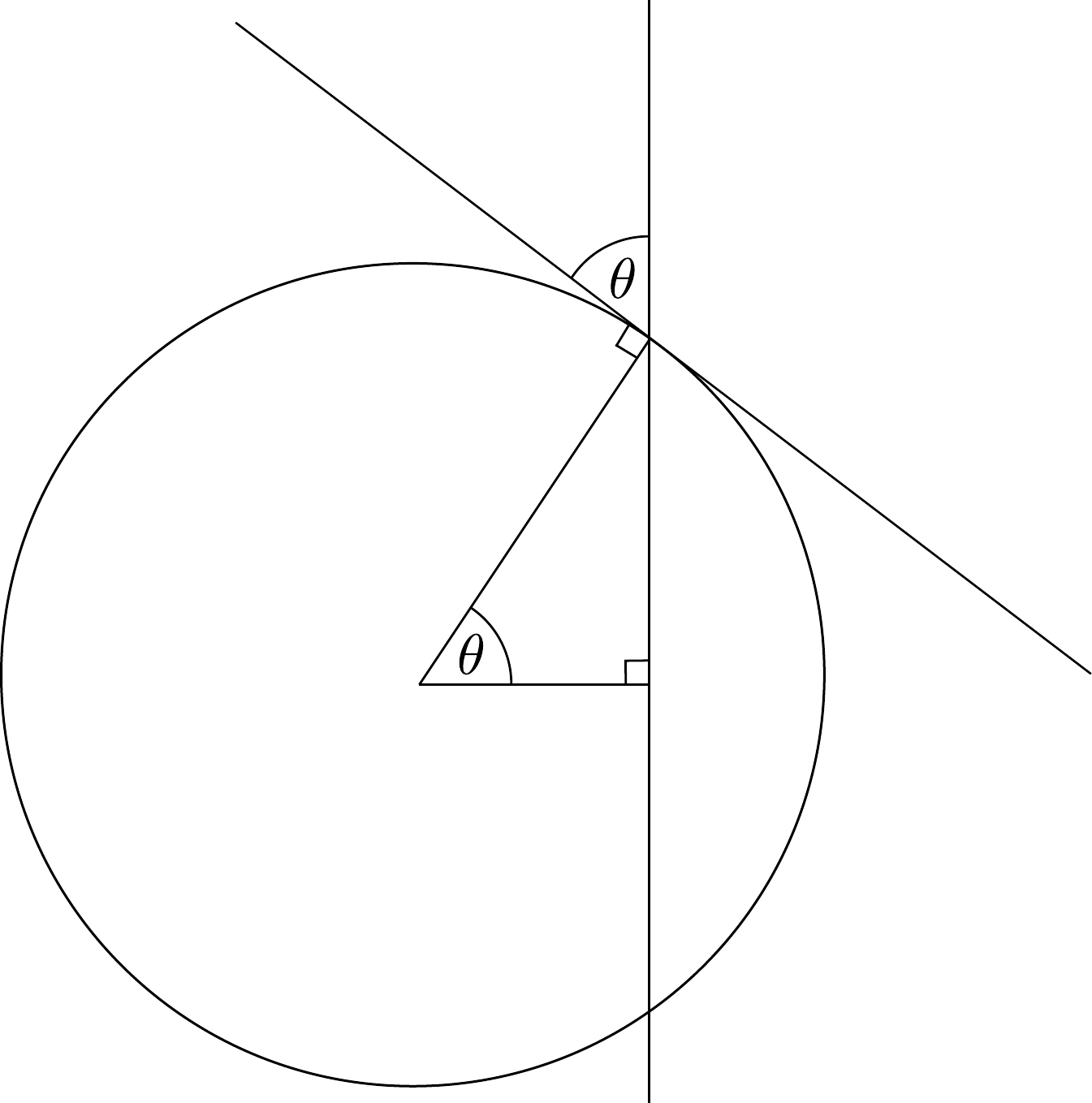}
\caption{The vertical line is $x=r\cos{( \theta )}$}
\label{fig:thetaangle}
\end{center}
\end{figure}

\begin{proof}After translating and rotating if necessary, we may assume that the line is vertical, and further that it is of the form $x=a$ for $a\geq0$, and that the circle is centered at the origin. The intersection points are then $(a, \sqrt{r-a^2})$ and $(a, -\sqrt{r-a^2})$; we focus on the former. The angle $\theta$ is the angle at which the line $x=a$ and the tangent line to the circle at this point intersect (see Figure \ref{fig:thetaangle}). Since the radius of the unit circle meets the tangent line at a right angle, the triangle with vertices at $(0,0)$, $(a,0)$ and $(a, \sqrt{r-a^2})$ has an angle of $\frac{\pi}{2} - \theta$ at $(a, \sqrt{r-a^2})$. This triangle has a right angle at $(a,0)$, and so must have angle $\theta$ at $(0,0)$. Since the hypotenuse of the triangle is a radius of the unit circle, it follows that $a=\cos{( \theta )}$.\end{proof}

When finding the precise location of the prism in the upper half-space, we will need to find three lines (corresponding to vertical planes which ``meet" at the ideal vertex) and two circles (corresponding to non-vertical planes). We will take one circle to be the unit circle and then find the three lines based on that. One of the lines will be taken to be $x=c$ for some constant $c$; the other two lines will have the form $y=mx+b$, with slope $m$ and $y$-intercept $b$. In order to find the second circle, we will find three equations, corresponding to the three angles at which that plane meets three other planes, in the three unknowns $(x_0,y_0)$ and $r$, representing, respectively, the coordinates of the center, and the radius, of the second circle. This circle will meet the unit circle and two lines; in some circumstances, we will need an equation coming from the fact that the second circle meets a given line $y=mx+b$ at a prescribed angle $\phi$.

\begin{lemma}\label{lem:intersectline}If a circle with center $(x_0,y_0)$ and radius $r$ meets the line $y=mx+b$ at angle $\phi$, and $(x_0, y_0)$ lies on or above the line (that is, $y_0 \geq mx_0 + b$), then $x_0,y_0$ and $r$ satisfy the equation
\[y_0 - \frac{r\cos{( \phi )} }{\sqrt{m^2+1}} = m\left( x_0 + \frac{mr\cos{( \phi )}}{\sqrt{m^2+1}} \right) +b.\]   \end{lemma}

\begin{figure}[htb]
\begin{center}
\includegraphics[scale=0.55]{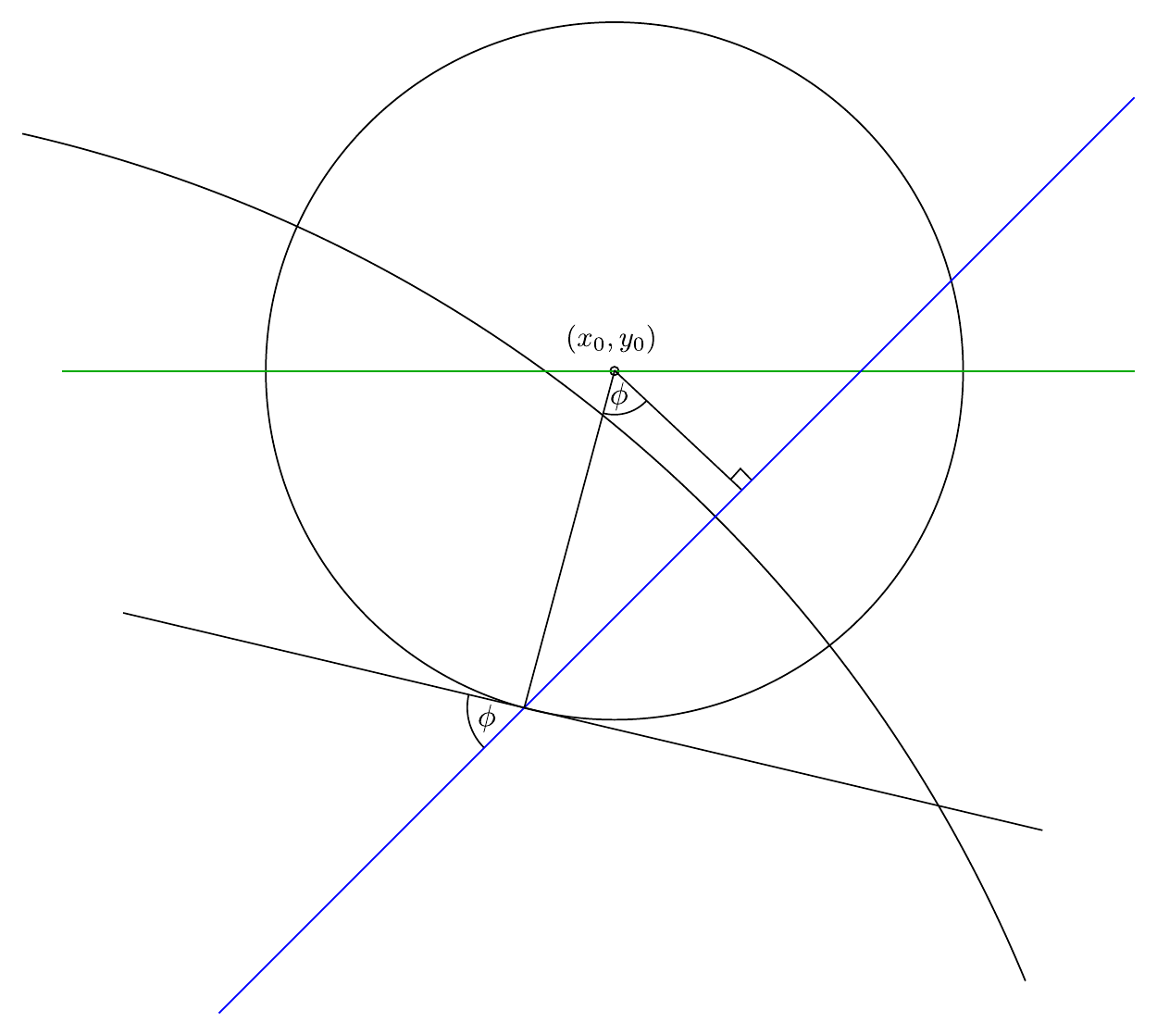}
\caption{The circle has center $(x_0, y_0)$ and radius $r$; the other arc is part of the unit circle}
\label{fig:lastcircle}
\end{center}
\end{figure}

\begin{proof}The vector $\langle 1, m \rangle$ is parallel to the line, so the vector $\langle m, -1 \rangle$ is perpendicular to the line. The unit vector parallel to this is
\[ \left\langle \frac{m}{\sqrt{m^2+1}}  , \frac{-1}{\sqrt{m^2+1}}. \right\rangle \]
By trigonometry (see Figure \ref{fig:lastcircle}), we see that the point which is distance $r\cos{( \phi )}$ away from $(x_0, y_0)$ in the direction of this vector lies on the line. Therefore, plugging the $x$- and $y$-coordinates of the vector
\[ \left\langle x_0 + r\cos{( \phi )}\frac{m}{\sqrt{m^2+1}}, y_0 - r\cos{( \phi )}\frac{1}{\sqrt{m^2+1}}  \right\rangle \]
into the formula $y=mx+b$ yields the required equation.\end{proof}

We will obtain another of our three equations by using the angle $\phi$ at which the second circle meets the unit circle.

\begin{lemma}\label{lem:cosines}If a circle with center $(x,y)$ and radius $r$ meets the unit circle at angle $\phi$, then $x,y$, and $r$ satisfy the equation
\[ x^2+y^2 = 1 + r^2 + 2\cos{( \phi )}r. \]\end{lemma}

\begin{proof} This is an application of the Cosine Law to the triangle whose vertices are at $(0,0)$, $(x,y)$, and one of the points where the circles intersect (see Figure \ref{fig:cosines}). This triangle has side lengths, $1$, $r$ and $\sqrt{x^2+y^2}$. The angle opposite the latter side is $\pi - \phi$ because the angle between the respective tangent lines at this vertex is $\phi$, and the two angles between the tangent lines and their respective radii are both $\pi/2$. The equation follows from the fact that $\cos{(\pi - \phi)} = -\cos{(\phi)}$. \end{proof}

\begin{figure}[htb]
\begin{center}
\includegraphics[scale=0.55]{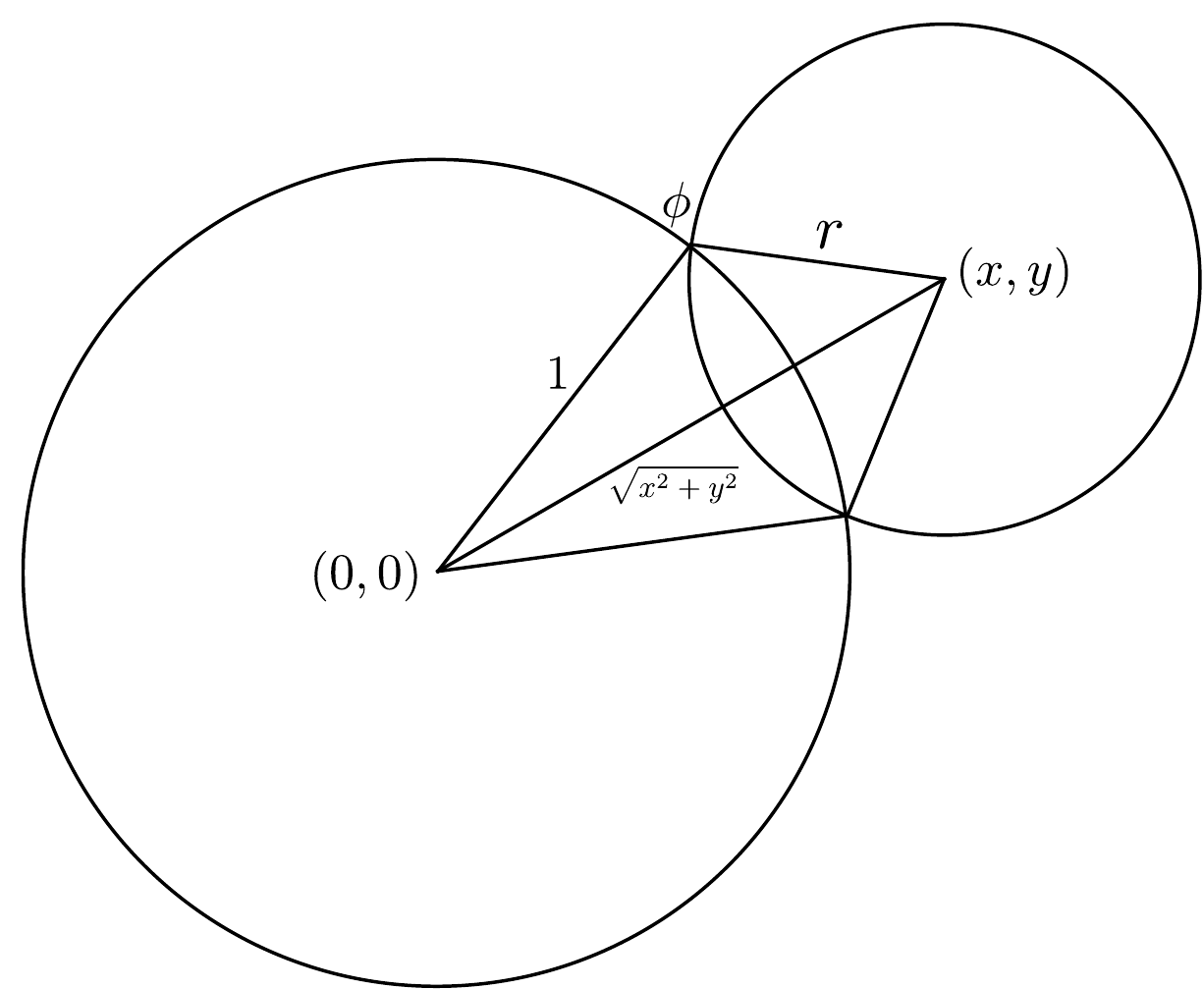}
\caption{The circles meet at angle $\phi$}
\label{fig:cosines}
\end{center}
\end{figure}

\section{Prisms}\label{section:prisms}

In this section, we will describe the dihedral angles of all hyperbolic triangular prisms with one ideal vertex, and a way to construct each prism in the upper half-space model of $\Hys$. There, the dihedral angles will be the angles at which the vertical and hemispherical geodesic planes intersect. We will find these planes by considering the lines and circles where they intersect with the plane $\{ (x,y,z) \mid  z=0 \}$, where we find lines and circles which must intersect at the same prescribed angles.

Such a prism is specified by nine positive integers, which we will denote as $a_1$ through $a_9$, corresponding to dihedral angles $\pi/a_i$. We label the prism as in Figure \ref{fig:LabeledPrism}.

\begin{figure}[htb]
\begin{center}
\includegraphics[scale=0.6]{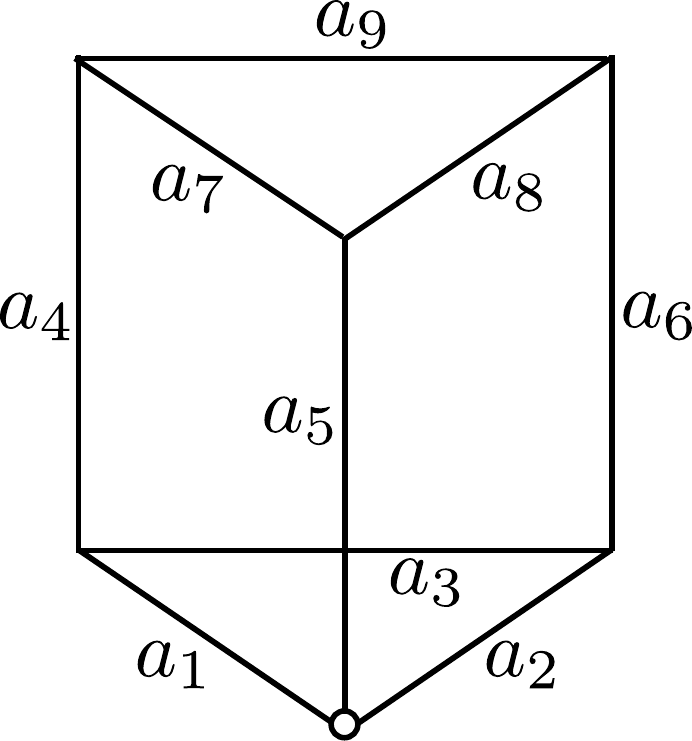}
\caption{The labels $a_1$ through $a_9$}
\label{fig:LabeledPrism}
\end{center}
\end{figure}

We note that due to the reflectional symmetry of the prism, once we have treated one prism, we have also treated the prism one obtains by exchanging the pairs $a_1$ and $a_2$, $a_4$ and $a_6$, and $a_7$ and $a_8$.

There are restrictions on the combinations of values taken by the labels $a_i$ which correspond to the conditions given in Theorem \ref{andreev}. Specifically:
\begin{itemize}\item condition 1 of Theorem \ref{andreev} means that all labels must be positive (i.e. not $0$ or $\infty$; we assume this of the $a_i$); 
\item condition 2 states that at the five non-ideal vertices, the three edges incident to the vertex must have the labels of a spherical triangle, and we add here that the three edges incident to the ideal vertex must have the labels of a Euclidean triangle; and
\item condition 3 states that labels $a_4$, $a_5$ and $a_6$ must be the labels of a hyperbolic triangle.
\end{itemize}
We disregard the other two conditions: condition $4$ does not apply because the prism has no prismatic $4$-circuits; and any labeling of the prism which meets the stated conditions will already meet condition $5$. This is because all dihedral angles in question are at most $\pi/2$, at most one of $a_4$, $a_5$ and $a_6$ can be $2$ (the others must be larger) from condition 3, and at most one of $a_1$ and $a_2$ may be $2$ from the ideal vertex condition.

%

\subsection{[2, 3, 6] Cusp}  

\begin{lemma}There are eight infinite familes and 32 specific configurations with labels 2, 3, and 6 at the ideal vertex.\end{lemma}

\begin{proof}Here $a_1$, $a_2$ and $a_5$ take the values $2, 3$ and $6$. We first note that $a_5 \neq 2$. This is because whichever of $a_4$ or $a_6$ labels an edge which meets the edge labeled $6$ must take the value $2$, and then $a_4$, $a_5$ and $a_6$ are not the labels of a hyperbolic triangle.

Let $a_5 = 3$. Then, by symmetry, without loss of generality we suppose $a_1=2$ and $a_2=6$. Then $a_3 = a_6=2$. Since $a_6=2$ and $a_5=3$, we must have $a_4 \geq 7$, and thus that $a_7 = a_9 = 2$. The remaining label $a_8$ may take the values $2, 3, 4$ or $5$. Each of these 4 cases gives us one infinite family of labelings, indexed by $n \geq 7$ which corresponds to the value of $a_4$.

Now suppose $a_5=6$, and without loss of generality $a_1=2$ and $a_2=3$. Since $a_5 = 6$, we must have $a_7=a_8=2$. If $a_3 = 4$ then we must have $a_6 = 2$ and $a_4 \leq 3$, in which case $a_4$, $a_5$ and $a_6$ are not the labels of a hyperbolic triangle. A similar argument applies if $a_3 > 4$. If $a_3 = 3$, then we must have $a_6 = 2$, and then $a_4$ must be $4$ or $5$. In each of these two cases, $a_9$ may be $2$ or $3$. 

Finally, if again $a_5=6$, $a_1=2$, $a_2=3$, and $a_7=a_8=2$, it remains to consider $a_3 = 2$. If $a_6 = 2$, then $a_4$ could be 4 or 5 -- in each case $a_9$ is either $2$ or $3$ -- or $a_4 \geq 6$, in which case $a_9 = 2$. If $a_6 = 3$, then $a_4$ could be 3, 4 or 5 -- if $a_4 = 3$, $a_9$ is 2, 3, 4 or 5; if $a_4$ is 4 or 5, $a_9$ is $2$ or $3$ -- or $a_4 \geq 6$, in which case $a_9 = 2$. If $a_6 = 4$, then $a_4$ could be 2, 3, 4 or 5 -- in each case $a_9$ is either $2$ or $3$ -- or $a_4 \geq 6$, in which case $a_9 = 2$. If $a_6 = 5$, then $a_4$ could be 2, 3, 4 or 5 -- in each case $a_9$ is either $2$ or $3$ -- or $a_4 \geq 6$, in which case $a_9 = 2$.\end{proof}

We will compute explicitly the location of the prism in upper half-space for one label arrangement, where $a_1=2$, $a_2=6$, and $a_5=3$. Let $a_3=2$, $a_4=7$, $a_6=2$, $a_7=2$, $a_8=3$, and $a_9=2$. This is shown in Figure \ref{fig:TriangularPrism}.
\begin{figure}[htb]
\begin{center}
\includegraphics[width=0.2\linewidth]{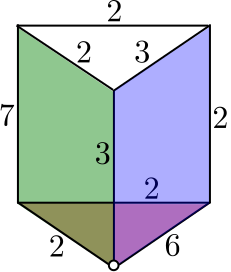}
\caption{The left quadrilateral face is green; the right is blue; and the lower triangular face is red}
\label{fig:TriangularPrism}
\end{center}
\end{figure}
The values for all the edges were specifically chosen to satisfy definitions \ref{def:euclidean}, \ref{def:spherical}, and \ref{def:hyperbolic}.  Using these definitions, we can examine different arrangements of the Euclidean vertex with values of 2, 3, and 6. 


We will refer to the left quadrilateral face, whose edges have labels $a_1$, $a_4$, $a_7$ and $a_5$, as the red face. The left quadrilateral face, with edges $a_2$, $a_6$, $a_9$ and $a_5$, will be the blue face. The lower triangular face, with edges $a_1$, $a_2$ and $a_3$, will be the red face. These faces all meet at the ideal vertex, and hence correspond to vertical lines in the horizontal plane.

Making the back quadrilateral face correspond to the unit circle, we choose the red face to correspond to the line $x=0$, which meets the unit circle at a right angle. By Lemma \ref{lem:thetaline}, the green face is the line $y=\cos{(\pi / 7 )}$, which meets the red face at a right angle, and the unit circle at $\pi /7$. The blue face is the line $y=\sqrt{3}x$, which meets the unit circle at a right angle and the red face at angle $\pi/6$  (see Figure \ref{fig:lines1}).


\begin{figure}[htb]
\begin{center}
\includegraphics[scale=0.4]{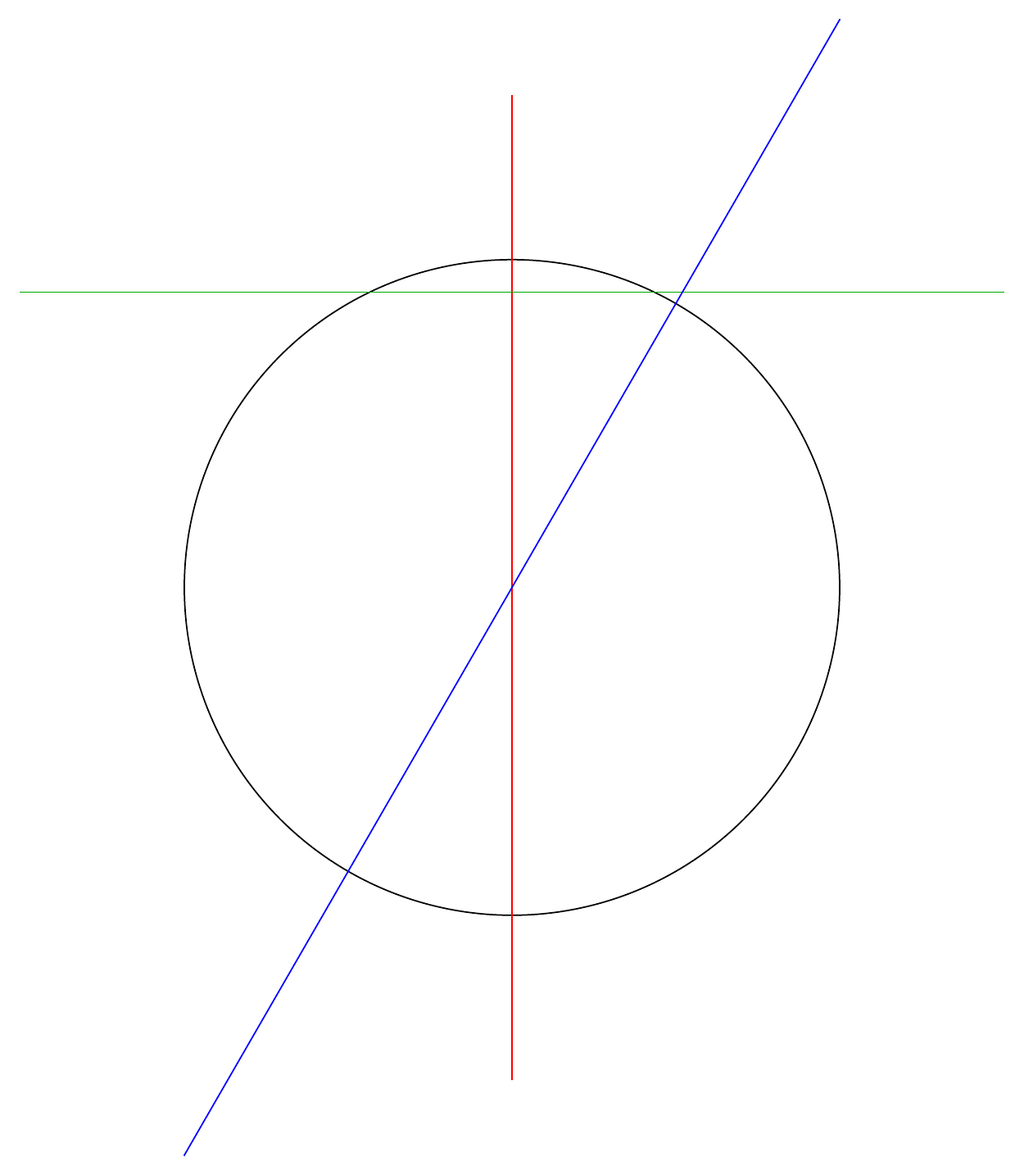}
\caption{Three lines and the unit circle}
\label{fig:lines1}
\end{center}
\end{figure}



It remains to find the last circle, corresponding to the top triangular face of the prism. Suppose this circle has center $(x,y)$ and radius $r$. This circle intersects the unit circle at an angle of $\pi/2$.  Using the Pythagorean Theorem to find an equation, we have that the equation of this intersection is 
\begin{equation}1 + r^2 = x^2 + y^2.\end{equation}
Since this circle meets the green line at a right angle, we see that the center must be on the green line, and hence that 
\begin{equation} y=\cos{( \pi /7 )}.\end{equation}
We need one more equation in $x$, $y$ and $r$, and this comes from the fact that the last circle meets the blue line at $\pi/3$. By Lemma \ref{lem:intersectline}, 


\begin{equation} y-r = \sqrt{3}x \end{equation}
(see Figure \ref{fig:lines2}).
\begin{figure}[htb]
\begin{center}
\includegraphics[scale=0.4]{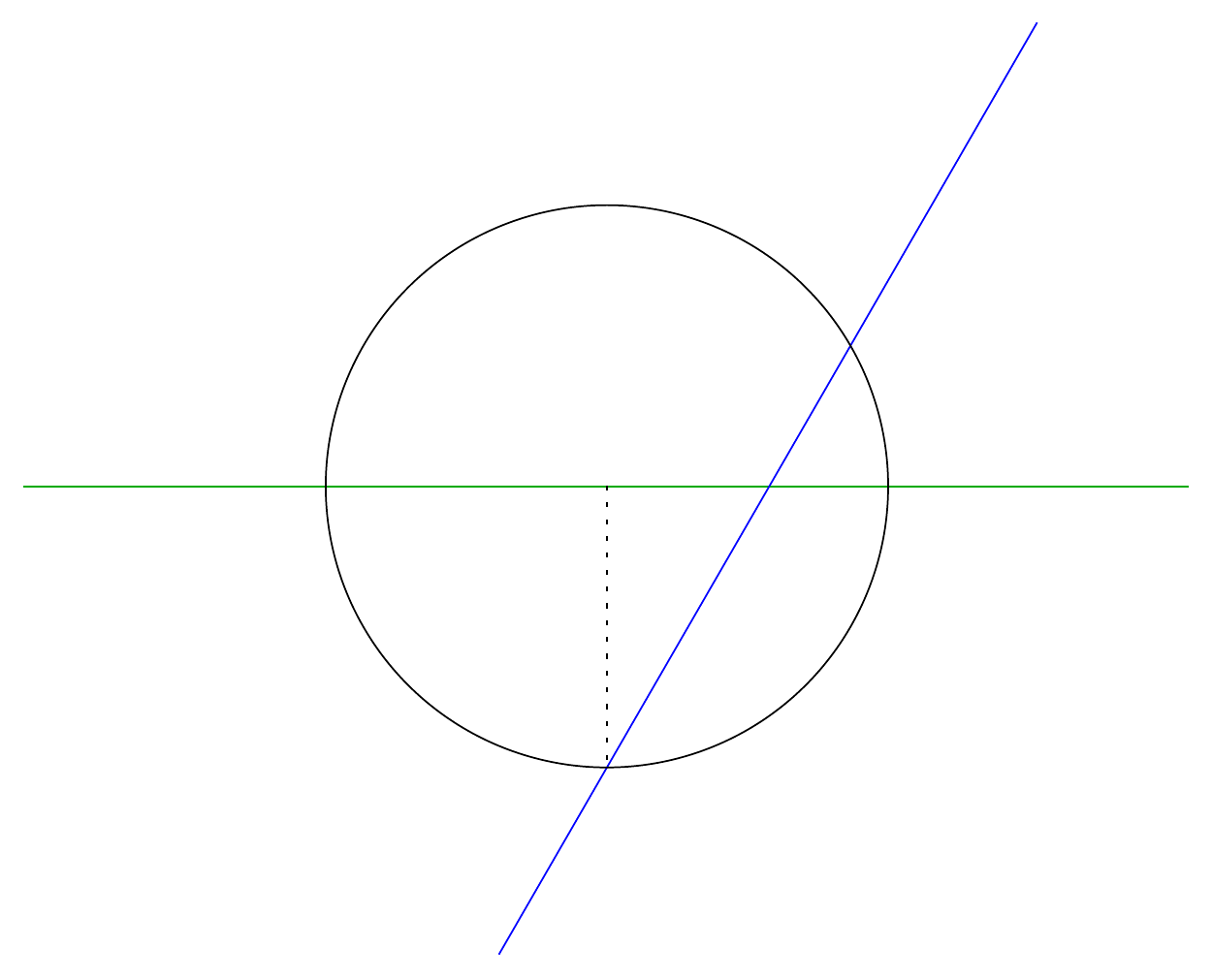}
\caption{The point $(x,y-r)$ lies on the blue line}
\label{fig:lines2}
\end{center}
\end{figure}

These equations, subject to $x>0$ and $r>0$, are sufficient to determine $x$, $y$ and $r$, and thus to determine the location of the prism precisely. We find that as well as $y=\cos{( \pi /7 )}$, that 
\[ x = \left(2\sqrt{3}\cos{(\pi/7)}-\sqrt{6\sin{(3\pi/14)}-2} \right)/4 \approx 0.4504\]
and 
\[r = \left( \sqrt{18\sin{(3\pi/14)} -6} -2\cos{(\pi/7)} \right) /4 \approx 0.1209.\]

Now that we have found the equations of the three lines and two circles which intersect at the prescribed angles, the hyperbolic prism we seek is the region inside the triangle defined by the three straight lines, and exterior to the two spheres whose equators are the given circles (see Figure \ref{fig:prism3d}).

\begin{figure}[htb]
\begin{center}
\includegraphics[scale=0.6]{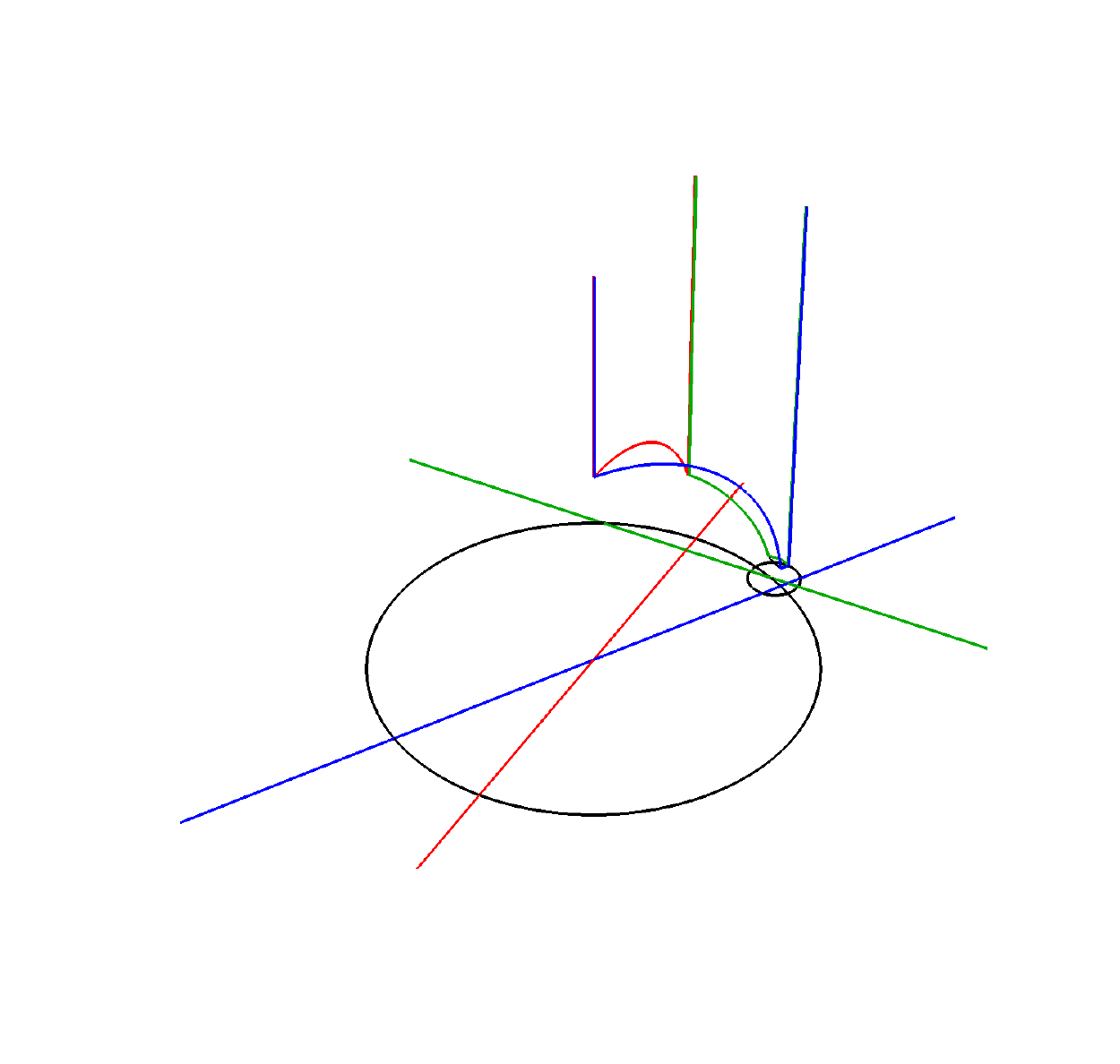}
\caption{The lines and circles define the hyperbolic planes bounding the prism}
\label{fig:prism3d}
\end{center}
\end{figure}


Changing the value of $a_4$ edge from 7 to 8, the equation of the green line changes to $y = \cos(\pi/8)$ by Lemma \ref{lem:thetaline}.  And the same thing happens changing it to equal 9.  This result shows that the line gets taller as the value of that side increases.  That edge can be left as $a_4 = m > 6$.

\subsection{[2, 4, 4] Cusp}  

\begin{lemma}There are four infinite familes and 24 specific configurations with labels 2, 4, and 4 at the ideal vertex.\end{lemma}

\begin{proof}First, we note that if $a_5=2$, then $a_1=a_2=4$. Then $a_4$ and $a_6$ are both at most 3, and then $a_4$, $a_5$ and $a_6$ are not the labels of a hyperbolic triangle. Thus we must have $a_5=4$, and without loss of generality we will assume $a_1=2$ and $a_4=4$.

With these assumptions in place, we next note that $a_3$ must be 2 or 3. Let us first treat the case $a_3=3$. In this case, we must have $a_6=2$, and then $a_4=5$, because it forms a hyperbolic triangle with $2$ and $4$ and a spherical triangle with $2$ and $3$. The possible labels for $(a_7, a_8, a_9)$ are then $(2,2,2), (2,2,3), (2,3,2), (2,3,3)$ and $(3,2,2)$.

Finally, suppose $a_1=2$, $a_2=4$, $a_3=2$, and $a_5=4$. Then $a_6$ is either $2$ or $3$. If $a_6=2$, then $a_4 \geq 5$. If $a_4 = 5$, then as above, the possible labels for $(a_7, a_8, a_9)$ are $(2,2,2), (2,2,3), (2,3,2), (2,3,3)$ and $(3,2,2)$. If $a_4 \geq 6$, then $a_7=a_9=2$, and $a_8$ may be $2$ or $3$. If $a_6=3$, then $a_4 \geq 3$. If $a_4 = 3$, then the possible labels for $(a_7, a_8, a_9)$ are $(2,2,2), (2,2,3), (2,2,4), (2,2,5), (2,3,2)$ and $(3,2,2)$. If $a_4 = 4$ or $5$, then the possible labels for $(a_7, a_8, a_9)$ are $(2,2,2), (2,2,3), (2,3,2)$ and $(3,2,2)$. If $a_4 \geq 6$, then $a_7=a_9=2$, and $a_8$ may be $2$ or $3$.\end{proof}

The first possibility of this arrangement we will work with is where $a_9$ is 3, $a_8$ is 2, $a_7$ is 2, $a_5$ is 4, the $a_4$ is $m \geq 5$, $a_6$ is 2.  This $a_1$ is 2, the $a_2$ is 4, and $a_3$ is 2 (see Figure \ref{fig:TriangularPrism3}).

\begin{figure}[htb]
\begin{center}
\includegraphics[width=0.2\linewidth]{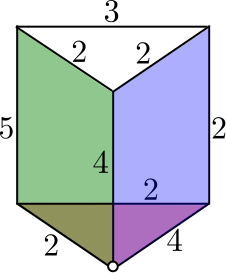}
\caption{A $[2, 4, 4]$ cusp}
\label{fig:TriangularPrism3}
\end{center}
\end{figure}

With the bottom vertex being 2, 4, 4, we have a $(\pi/2, \pi/4, \pi/4)$ triangle, with blue, red and green sides.  The back face of the prism meets green at an angle of $\pi/5$, and the back face meets the blue and red faces at $\pi/2$.  The back face corresponds to the unit circle. We choose the red face to correspond to the line $x=0$. By Lemma \ref{lem:thetaline}, the green face corresponds to the line $y=\cos{(\pi/5)}$. The blue line corresponds to the line $y=x$.

Finally, the top face meets blue and green at angle $\pi/2$, and the unit circle at angle $\pi/3$.  Meeting blue and green at $\pi/2$ means the center of the last circle is on both blue and green lines.  Their intersection point is $(x, y) = (\cos{(\pi/5)}, \cos{(\pi/5)})$. Meeting the unit circle at $\pi/3$ means that, by Lemma \ref{lem:cosines}, we have \[ x^2 + y^2 = 1^2 + r^2 -2\cos{\left( \frac{2\pi}{3} \right) r}\]
and so
\[  x^2 + y^2 = 1 + r^2 + r.\]
Here we find $r=(\sqrt[4]{5}-1)/2$.

As we saw, with the $[ 2, 4, 4] $ cusp the edges $a_7, a_8$ and $a_9$ are always labeled 2, 3 or 4, and they cannot all be labeled 3. We now describe what happens to the equations defining $x$, $y$ and $r$ when these labels change, noting that the red, blue and green faces, as well as the unit circle, are unaffected by these changes.

If we change $a_9$ to be 2, keeping $a_7$ and $a_8$ as 2.  This means that the top face intersects the back face at $\pi/2$.  The intersection between the last circle and the unit circle creates an angle of $\pi/2$. We still have that $x = y = \cos{( \pi / 5 )}$, and the third equation becomes $x^2 + y^2 = 1 + r^2$.  

If $a_9$ is kept as a 2, and we change $a_8$ to 3, then because the new circle meets the green line at right angles, we still have $y = \cos{( \pi / 5)}$. Also, since $a_9 = 2$, the new circle meets the unit circle at right angles, and so $x^2 + y^2 = 1 + r^2$. By Lemma \ref{lem:intersectline}, the final equation in this case is
\[ y = x + \frac{r\sqrt{2}}{2}. \]



The next arrangement we consider has $a_7$, $a_8$, and $a_9$ equal to 2, the vertical edges as $m \geq 5$, 4, and 2, and the bottom edges as $a_1=2$, $a_2=4$, and $a_3=3$. This change shifts the red face from intersecting the unit circle at $x = 0$ to $x = -1/2$. The green line will be $y = \cos{( \pi / a_4 )}$, and the blue will be $y=x$. Since $a_7 = a_8 = 2$, the center of the second circle is at the intersection of the blue and green lines. Letting $a_9$ be 2 results in the final equation of the last circle being $1 + r^2 = x^2 + y^2$.

Lastly, keeping $a_3$ as 3 (so the red line is still $x=-1/2$), we change $a_6$ to 3.  By Lemma \ref{lem:thetaline}, this shifts the blue line downward from $y = x$ to $y = x - \sqrt{2}/2$. The equations defining $x, y$ and $r$ here are $y=\cos{( \pi / a_4 )}$, $y = x - \sqrt{2}/2$ and $x^2+y^2 = 1+r^2$.



\subsection{[3, 3, 3] Cusp}  

\begin{lemma}There are no infinite familes and 22 specific configurations with labels 3, 3, and 3 at the ideal vertex.\end{lemma}

\begin{proof}In studying labelings here, we note that this labeling at the Euclidean vertex is symmetric, and thus we will discard some labelings as being symmetric to other labelings already listed.

We first note that if $a_3 >2$, then $a_4 = a_6 = 2$ and we do not have a hyperbolic triangle. So $a_3=2$. Then by spherical triangles, $a_4$ and $a_6$ must both be one of $2, 3, 4$ or $5$, but neither can be 2 because of the hyperbolic triangle, and also they cannot both be $3$. Because of symmetry considerations, we suppose without loss of generality that $a_4 \leq a_6$. 

\begin{itemize}\item If $a_4=3$ and $a_6=4$, then $(a_7,a_8,a_9)$ can be $(2,2,2), (2,2,3), (2,3,2), (3,2,2), (4,2,2)$ or $(5,2,2)$;
\item If $a_4=3$ and $a_6=5$, then $(a_7,a_8,a_9)$ can be $(2,2,2), (2,2,3), (2,3,2), (3,2,2), (4,2,2)$ or $(5,2,2)$;
\item If $a_4=4$ and $a_6=4$, then $(a_7,a_8,a_9)$ can be $(2,2,2), (2,2,3)$ or $(2,3,2)$ (we discard $(3,2,2)$ as it is symmetric to $(2,3,2)$);
\item If $a_4=4$ and $a_6=5$, then $(a_7,a_8,a_9)$ can be $(2,2,2), (2,2,3), (2,3,2)$ or $(3,2,2)$;
\item If $a_4=5$ and $a_6=5$, then $(a_7,a_8,a_9)$ can be $(2,2,2), (2,2,3)$ or $(2,3,2)$ (we discard $(3,2,2)$ as it is symmetric to $(2,3,2)$).
\end{itemize}
\end{proof}

The arrangement we examine is with $a_6$ equal to 4, $a_4= 3$, and $a_7$, $a_8$, and $a_9$ equal to 2. Thus the full arrangement is $[ a_1,  a_2,  a_3,  a_4,  a_5,  a_6,  a_7,  a_8,  a_9, ] = [3, 3, 2, 3, 3, 4, 2, 2, 2 ]$. As before, we will let the back side of the prism be the unit circle. Since all of the admissible labelings here have $a_3=2$, we will fix the red line to be $x=0$.



Since the blue side meets the back side at $\pi/4$, the blue line intersects the unit circle at $\pi/4$.  The green side meets the back side at $\pi/3$, and the red side meets the back side at $\pi/2$.  This means that the red side creates a line that passes through the center of the unit circle at angle $\pi/2$.  The red, green, and blue lines create an equilateral triangle since their sides meet each other at $\pi/3$.  


If we shift the green line so that is passes through the center of the unit circle, we can find the slope of the blue line.  Since this line creates an angle of $\pi/3$, in the fourth quadrant, there is a remaining $\pi/6$ in order to make a right angle.  Thus the slope of the blue line is 
\[ \tan \left( \frac{\pi}{6} \right) = \frac{\sin(\pi/6)}{\cos(\pi/6)} = \frac{-1/2}{\sqrt{3}/2} = -\frac{\sqrt{3}}{3}.\]

The slope $-\sqrt{3}/3$ corresponds to the vector $\langle 1, -\sqrt{3}/3 \rangle$.  To find a vector orthogonal to this vector, we flip one value and negate the other value.  We get the vector $\langle 1, \sqrt{3} \rangle$.  The unit vector is $\langle 1/2, \sqrt{3}/2\rangle.$  Using Lemma \ref{lem:thetaline}, we multiply the unit vector by $\cos(\pi/3)$ to get the point $(1/4, \sqrt{3}/4)$.  Plugging in the slope and this point into the point-slope formula, we get $y = (-\sqrt{3}/3)x + \sqrt{3}/3.$

The slope of the blue line is $\sqrt{3}/3.$  The vector corresponding to this slope is $\langle 1, -\sqrt{3} \rangle$ and the unit vector is $\langle 1/2, -\sqrt{3}/2 \rangle$.  We multiply the unit vector by $\cos(\pi/4)$ to find the point $(\sqrt{2}/4, -\sqrt{6}/4)$.  Plugging this point and the slope into the point-slope formula, we get
$y = (\sqrt{3}/3)x - \sqrt{6}/3.$

Since $a_7$, $a_8$, and $a_9$ equal 2, that means the last circle intersects the unit circle at right angles.  This also means that the center of this last circle lies on the intersection of the green and blue lines.  Since the last circle intersects the unit circle at right angles, we use the Pythagorean Theorem to find that the last equation is $1 + r^2 = x^2 + y^2.$





Here is a list of the equations of the following arrangements where we adjust the values of $a_4$ and $a_6$.

\begin{center}
	\begin{tabular}{ c c c }
		\text{Values of $a_4$, $a_5$ and $a_6$: Left to Right} & \text{Green Line} & \text{Blue Line} \\
		3, 3, 4 & $y = -\frac{\sqrt{3}}{3}x + \frac{\sqrt{3}}{3}$ & $y = \frac{\sqrt{3}}{3}x - \frac{\sqrt{6}}{3}$ \\
		3, 3, 5 & $y = -\frac{\sqrt{3}}{3}x + \frac{\sqrt{3}}{3}$ & $y = \frac{\sqrt{3}}{3}x - \frac{2\sqrt{3}}{3}\cos(\frac{\pi}{5})$ \\
		4, 3, 4 & $y = -\frac{\sqrt{3}}{3}x + \frac{\sqrt{6}}{3}$ & $y = \frac{\sqrt{3}}{3}x - \frac{\sqrt{6}}{3}$ \\
		4, 3, 5 & $y = -\frac{\sqrt{3}}{3}x + \frac{\sqrt{6}}{3}$ & $y = \frac{\sqrt{3}}{3}x - \frac{2\sqrt{3}}{3}\cos(\frac{\pi}{5})$ \\
		5, 3, 5 & $y = -\frac{\sqrt{3}}{3}x + \frac{2\sqrt{3}}{3}\cos(\frac{\pi}{5})$ & $y = \frac{\sqrt{3}}{3}x - \frac{2\sqrt{3}}{3}\cos(\frac{\pi}{5})$ \\
	\end{tabular}
\end{center}

As long as $a_7 = a_8 = a_9 = 2$, the third equation defining the last circle that intersects the unit circle does not change for all of these specific arrangements.  Thus, for all of these arrangements, the equation of the last circle is $1 + r^2 = x^2 + y^2.$ When we change $a_9$ to 3, the equation of the last circle for the arrangements listed above is $x^2 + y^2 = 1 + r^2 + r.$ When $a_7$ or $a_8$ change, the the equations change according to Lemma \ref{lem:intersectline} and Lemma \ref{lem:cosines}.

\section{Matrices}\label{s:matrices}

In this section, we describe a general method which produces lines and circles which intersect at the prescribed angles, which produces the same results described in the previous section. We then show how to take this geometric data of the lines and circles and use it to produce matrix generators in $\mathrm{PSL}_2(\mathbb{C})$ for the orientation-preserving subgroup of the group generated by reflections in the faces of the corresponding prism. Each group will be generated by four matrices, where each matrix acts by pairing two faces of the polyhedron one obtains by doubling the prism across one face. 

We note at this point that this method also works for examples of prisms which have more than one ideal vertex; that is, where one or more of the five vertices not located at $\infty$ in the upper half-space is found in the complex plane. In this paper, we restrict to the case of prisms with one cusp, but the method we describe should work for some, if not all, labelings of the prism which correspond to prisms with more than one cusp.

\subsection{The case $a_3 = 2$}

As above, we suppose that the back quadrilateral face lies on the unit sphere, or equivalently that one of the two circles is the unit circle. We further suppose that the red face lies above the imaginary axis, or equivalently that one of the straight lines is $x=0$; this corresponds to asking that $a_3=2$. Lastly, we assume that the polyhedron lies to the right of the imaginary axis as we view it from above; in other words, we assume that it lies in the region of $\Hys$ with $x \geq 0$.

By considering Figure \ref{fig:thetaangle} and applying Lemma \ref{lem:thetaline}, we see that the blue line has equation
\[ y = \cot{ \left( \frac{\pi}{a_2} \right) } x - \frac{\cos{( \pi/a_6 )}}{\sin{( \pi / a_2 )}} \] 
and, by similar reasoning, the green line has equation
\[ y = -\cot{ \left( \frac{\pi}{a_1} \right) } x + \frac{\cos{( \pi/a_4 )}}{\sin{( \pi / a_1 )}}. \]
The final face is on a circle with center $(x,y)$ and radius $r$, which meets the blue line at angle $\pi / a_8$. By Lemma \ref{lem:intersectline} we have one equation
\[y - \frac{r\cos{( \pi / a_8 )} }{\sqrt{\cot^2{( \pi / a_2 )}+1}} = \cot{( \pi / a_2 )} \left( x + \frac{r\cot{( \pi / a_2 )}\cos{( \pi / a_8 )}}{\sqrt{\cot^2{( \pi / a_2 )}+1}} \right) - \frac{\cos{( \pi / a_6 )}}{ \sin{( \pi / a_2 )} },\]
which simplifies to
\begin{equation}\label{eq1} y - r\cos{( \pi / a_8 )} \sin{( \pi / a_2 )} = \cot{( \pi / a_2 )} \left( x + r\cos{( \pi / a_2 )}\cos{( \pi / a_8 )} \right) - \frac{\cos{( \pi / a_6 )}}{ \sin{( \pi / a_2 )} }.\end{equation}
By applying Lemma \ref{lem:intersectline}, with appropriate modifications, to the green line, we also have the equation
\begin{equation}\label{eq2}y + r\sin{( \pi / a_1 )} \cos{( \pi / a_7 )} = -\cot{( \pi / a_1 )} \left( x + r\cos{( \pi / a_1 )}\cos{( \pi / a_7 )} \right) + \frac{\cos{( \pi / a_4 )}}{ \sin{( \pi / a_1 )} }.\end{equation}  
Finally, the two circles intersecting at angle $\pi / a_9$ yields, via the Cosine Law and Lemma \ref{lem:cosines}, the equation
\[ x^2 + y^2  = 1^2 + r^2 - 2(1)(r)\cos{ \left( \pi - \frac{\pi}{a_9} \right)} \]
or
\begin{equation}\label{eq3} x^2 + y^2  = 1^2 + r^2 + 2r\cos{\left( \frac{\pi}{a_9} \right)}.\end{equation}
Equations \ref{eq1}, \ref{eq2} and \ref{eq3} together define $x,y$ and $r$ and determine the final circle required.

With all of this work in mind, we require seven quantities to write down the matrices we seek. These quantities are $y_1 = \cos{( \pi / a_4 )} / \sin{( \pi / a_1 )}$ and $y_2 = - \cos{( \pi / a_6 )} / \sin{( \pi / a_2 )}$, the $y$-intercepts of the green and blue line; the angles $\theta_1 = \pi / a_1$ and $\theta_2 = \pi / a_2$ of the ideal triangle at these points; and the center $(x,y)$ and radius $r$ of the second circle. Given these quantities, the matrices are 
\[ M_1 = \begin{pmatrix} 0 & -1 \\ 1 & 0 \end{pmatrix}, \]
which pairs two sides which both lie on the unit sphere;
\[ M_2 = \begin{pmatrix} e^{-i\theta_1} & y_1 i (e^{i\theta_1} - e^{-i\theta_1}) \\ 0 & e^{i\theta_1} \end{pmatrix}, \]
which rotates counter-clockwise by angle $2\theta_1$ about $(0,y_1)$;
\[ M_3 = \begin{pmatrix} e^{i\theta_2} & y_2 i (e^{-i\theta_2} - e^{i\theta_2}) \\ 0 & e^{-i\theta_2} \end{pmatrix}, \]
which rotates clockwise by angle $2\theta_2$ about $(0,y_2)$; and
\[ M_4 = \begin{pmatrix} \frac{1}{r}(-x+yi) & \frac{1}{r}(x^2+y^2)-r \\ \frac{1}{r} & \frac{1}{r}(-x-yi) \end{pmatrix}, \]
which sends the second circle to its reflection in the imaginary axis. These matrices satisfy the relations
\[ M_2^{a_1} = 1, \ \ M_3^{a_2} = 1, \ \ M_1^{a_3} = M_1^2 = 1, \ \ ( M_2^{-1} M_1)^{a_4} = 1, \ \ (M_3^{-1} M_2)^{a_5} = 1, \]
\[ (M_3^{-1} M_1)^{a_6} = 1, \ \ (M_4^{-1} M_2)^{a_7} = 1, \ \ (M_4^{-1} M_3)^{a_8} = 1, \ \ (M_4^{-1}M_1)^{a_9} = 1.\]

\subsection{The case $a_3 = 3$}

In the case that $a_3\neq 2$, then we saw that $a_3=3$. In this case, we keep the back face as corresponding to the unit circle, and move the red line to $x=-1/2$ so that it intersects the unit circle at angle $\pi/3$. The equations of the green and blue lines will be the same as the case $a_3=2$, and the second circle will be defined by the same three equations \ref{eq1}, \ref{eq2} and \ref{eq3}. As in the previous case, we define $\theta_1 = \pi / a_1$ and $\theta_2 = \pi / a_2$, and let $(x,y)$ and $r$ be the center and radius of the second circle. In place of $y_1$ and $y_2$ we define $z_1 = -1/2 + (\cos{( \pi / a_4 )} / \sin{( \pi / a_1 )} + \cot{( \pi /a_1 )}/2)i$ and $z_2 = -1/2 + (\cos{( \pi / a_6 )} / \sin{( \pi / a_2 )} - \cot{( \pi /a_2 )}/2)i$, the points where the green and blue lines meet the red line $x=-1/2$. Our matrices are then
\[ M_1 = \begin{pmatrix} -1 & -1 \\ 1 & 0 \end{pmatrix}, \]
which pairs two sides which lie on the unit sphere and the unit sphere centered at $(-1,0)$;
\[ M_2 = \begin{pmatrix} e^{-i\theta_1} & z_1(e^{i\theta_1} - e^{-i\theta_1}) \\ 0 & e^{i\theta_1} \end{pmatrix}, \]
which rotates counter-clockwise by angle $2\theta_1$ about $z_1$;
\[ M_3 = \begin{pmatrix} e^{i\theta_2} & z_2(e^{-i\theta_2} - e^{i\theta_2}) \\ 0 & e^{-i\theta_2} \end{pmatrix}, \]
which rotates clockwise by angle $2\theta_2$ about $z_2$; and
\[ M_4 = \begin{pmatrix} \frac{1}{r}(-(x+1)+yi) & \frac{1}{r}(-(x+1)+yi)(-x-yi)-r \\ \frac{1}{r} & \frac{1}{r}(-x-yi) \end{pmatrix}, \]
which sends the second circle to its reflection in the line $x=-1/2$. These matrices satisfy the relations
\[ M_2^{a_1} = 1, \ \ M_3^{a_2} = 1, \ \ M_1^{a_3} = M_1^3 = 1, \ \ ( M_2^{-1} M_1)^{a_4} = 1, \ \ (M_3^{-1} M_2)^{a_5} = 1, \]
\[ (M_3^{-1} M_1)^{a_6} = 1, \ \ (M_4^{-1} M_2)^{a_7} = 1, \ \ (M_4^{-1} M_3)^{a_8} = 1, \ \ (M_4^{-1}M_1)^{a_9} = 1.\]

\section{Results}\label{s:results}

In this section, we list all of the possible labeling of the prism which were described in Section 3. We note that the twelve infinite families, where one label may vary, are shaded gray to distinguish them.

\subsection{$[2,3,6]$ cusp}

\subsubsection{$a_3=3$}

\begin{center}
\begin{tabular}{|c|c|c|c|c|c|c|c|c|}
\hline
$a_1$ &  $a_2$ &  $a_3$ &  $a_4$ &  $a_5$ &  $a_6$ &  $a_7$ &  $a_8$ &  $a_9$  \\
\hline
\hline
2 & 3 & 3 & 4 & 6 & 2 & 2 & 2 & 2  \\
\hline
2 & 3 & 3 & 4 & 6 & 2 & 2 & 2 & 3  \\
\hline
2 & 3 & 3 & 5 & 6 & 2 & 2 & 2 & 2  \\
\hline
2 & 3 & 3 & 5 & 6 & 2 & 2 & 2 & 3  \\
\hline
\end{tabular}
\end{center}

\subsubsection{$a_3=2$}

\begin{tabular}{|c|c|c|c|c|c|c|c|c|}
\hline
$a_1$ &  $a_2$ &  $a_3$ &  $a_4$ &  $a_5$ &  $a_6$ &  $a_7$ &  $a_8$ &  $a_9$  \\
\hline
\hline
\rowcolor{Gray}
2 & 6 & 2 & $n \geq 7$ & 3 & 2 & 2 & 2 & 2  \\
\hline
\rowcolor{Gray}
2 & 6 & 2 & $n \geq 7$ & 3 & 2 & 2 & 3 & 2  \\
\hline
\rowcolor{Gray}
2 & 6 & 2 & $n \geq 7$ & 3 & 2 & 2 & 4 & 2  \\
\hline
\rowcolor{Gray}
2 & 6 & 2 & $n \geq 7$ & 3 & 2 & 2 & 5 & 2  \\
\hline
2 & 3 & 2 & 4 & 6 & 2 & 2 & 2 & 2 \\
\hline
2 & 3 & 2 & 4 & 6 & 2 & 2 & 2 & 3 \\
\hline
2 & 3 & 2 & 5 & 6 & 2 & 2 & 2 & 2 \\
\hline
2 & 3 & 2 & 5 & 6 & 2 & 2 & 2 & 3 \\
\hline
\rowcolor{Gray}
2 & 3 & 2 & $n \geq 6$ & 6 & 2 & 2 & 2 & 2 \\
\hline
2 & 3 & 2 & 3 & 6 & 3 & 2 & 2 & 2 \\
\hline
2 & 3 & 2 & 3 & 6 & 3 & 2 & 2 & 3 \\
\hline
2 & 3 & 2 & 3 & 6 & 3 & 2 & 2 & 4 \\
\hline
2 & 3 & 2 & 3 & 6 & 3 & 2 & 2 & 5 \\
\hline
2 & 3 & 2 & 4 & 6 & 3 & 2 & 2 & 2 \\
\hline
2 & 3 & 2 & 4 & 6 & 3 & 2 & 2 & 3 \\
\hline
2 & 3 & 2 & 5 & 6 & 3 & 2 & 2 & 2 \\
\hline
2 & 3 & 2 & 5 & 6 & 3 & 2 & 2 & 3 \\
\hline
\rowcolor{Gray}
2 & 3 & 2 & $n \geq 6$ & 6 & 3 & 2 & 2 & 2 \\
\hline
\end{tabular}
\quad
\begin{tabular}{|c|c|c|c|c|c|c|c|c|}
\hline
$a_1$ &  $a_2$ &  $a_3$ &  $a_4$ &  $a_5$ &  $a_6$ &  $a_7$ &  $a_8$ &  $a_9$  \\
\hline
\hline
2 & 3 & 2 & 2 & 6 & 4 & 2 & 2 & 2 \\
\hline
2 & 3 & 2 & 2 & 6 & 4 & 2 & 2 & 3 \\
\hline
2 & 3 & 2 & 3 & 6 & 4 & 2 & 2 & 2 \\
\hline
2 & 3 & 2 & 3 & 6 & 4 & 2 & 2 & 3 \\
\hline
2 & 3 & 2 & 4 & 6 & 4 & 2 & 2 & 2 \\
\hline
2 & 3 & 2 & 4 & 6 & 4 & 2 & 2 & 3 \\
\hline
2 & 3 & 2 & 5 & 6 & 4 & 2 & 2 & 2 \\
\hline
2 & 3 & 2 & 5 & 6 & 4 & 2 & 2 & 3 \\
\hline
\rowcolor{Gray}
2 & 3 & 2 & $n \geq 6$ & 6 & 4 & 2 & 2 & 2 \\
\hline
2 & 3 & 2 & 2 & 6 & 5 & 2 & 2 & 2 \\
\hline
2 & 3 & 2 & 2 & 6 & 5 & 2 & 2 & 3 \\
\hline
2 & 3 & 2 & 3 & 6 & 5 & 2 & 2 & 2 \\
\hline
2 & 3 & 2 & 3 & 6 & 5 & 2 & 2 & 3 \\
\hline
2 & 3 & 2 & 4 & 6 & 5 & 2 & 2 & 2 \\
\hline
2 & 3 & 2 & 4 & 6 & 5 & 2 & 2 & 3 \\
\hline
2 & 3 & 2 & 5 & 6 & 5 & 2 & 2 & 2 \\
\hline
2 & 3 & 2 & 5 & 6 & 5 & 2 & 2 & 3 \\
\hline
\rowcolor{Gray}
2 & 3 & 2 & $n \geq 6$ & 6 & 5 & 2 & 2 & 2 \\
\hline
\end{tabular}

\subsection{$[2,4,4]$ cusp}

\subsubsection{$a_3=3$}

\begin{center}
\begin{tabular}{|c|c|c|c|c|c|c|c|c|}
\hline
$a_1$ &  $a_2$ &  $a_3$ &  $a_4$ &  $a_5$ &  $a_6$ &  $a_7$ &  $a_8$ &  $a_9$  \\
\hline
\hline
2 & 4 & 3 & 5 & 4 & 2 & 2 & 2 & 2  \\
\hline
2 & 4 & 3 & 5 & 4 & 2 & 2 & 2 & 3  \\
\hline
2 & 4 & 3 & 5 & 4 & 2 & 2 & 3 & 2  \\
\hline
2 & 4 & 3 & 5 & 4 & 2 & 2 & 3 & 3  \\
\hline
2 & 4 & 3 & 5 & 4 & 2 & 3 & 2 & 2  \\
\hline
\end{tabular}
\end{center}

\subsubsection{$a_3=2$}

\begin{tabular}{|c|c|c|c|c|c|c|c|c|}
\hline
$a_1$ &  $a_2$ &  $a_3$ &  $a_4$ &  $a_5$ &  $a_6$ &  $a_7$ &  $a_8$ &  $a_9$  \\
\hline
\hline
2 & 4 & 2 & 5 & 4 & 2 & 2 & 2 & 2  \\
\hline
2 & 4 & 2 & 5 & 4 & 2 & 2 & 2 & 3  \\
\hline
2 & 4 & 2 & 5 & 4 & 2 & 2 & 3 & 2  \\
\hline
2 & 4 & 2 & 5 & 4 & 2 & 2 & 3 & 3  \\
\hline
2 & 4 & 2 & 5 & 4 & 2 & 3 & 2 & 2  \\
\hline
\rowcolor{Gray}
2 & 4 & 2 & $n \geq 6$ & 4 & 2 & 2 & 2 & 2  \\
\hline
\rowcolor{Gray}
2 & 4 & 2 & $n \geq 6$ & 4 & 2 & 2 & 3 & 2  \\
\hline
2 & 4 & 2 & 3 & 4 & 3 & 2 & 2 & 2  \\
\hline
2 & 4 & 2 & 3 & 4 & 3 & 2 & 2 & 3  \\
\hline
2 & 4 & 2 & 3 & 4 & 3 & 2 & 2 & 4  \\
\hline
2 & 4 & 2 & 3 & 4 & 3 & 2 & 2 & 5  \\
\hline
2 & 4 & 2 & 3 & 4 & 3 & 2 & 3 & 2  \\
\hline
\end{tabular}
\quad
\begin{tabular}{|c|c|c|c|c|c|c|c|c|}
\hline
$a_1$ &  $a_2$ &  $a_3$ &  $a_4$ &  $a_5$ &  $a_6$ &  $a_7$ &  $a_8$ &  $a_9$  \\
\hline
\hline
2 & 4 & 2 & 3 & 4 & 3 & 3 & 2 & 2  \\
\hline
2 & 4 & 2 & 4 & 4 & 3 & 2 & 2 & 2  \\
\hline
2 & 4 & 2 & 4 & 4 & 3 & 2 & 2 & 3  \\
\hline
2 & 4 & 2 & 4 & 4 & 3 & 2 & 3 & 2  \\
\hline
2 & 4 & 2 & 4 & 4 & 3 & 3 & 2 & 2  \\
\hline
2 & 4 & 2 & 5 & 4 & 3 & 2 & 2 & 2  \\
\hline
2 & 4 & 2 & 5 & 4 & 3 & 2 & 2 & 3  \\
\hline
2 & 4 & 2 & 5 & 4 & 3 & 2 & 3 & 2  \\
\hline
2 & 4 & 2 & 5 & 4 & 3 & 3 & 2 & 2  \\
\hline
\rowcolor{Gray}
2 & 4 & 2 & $n \geq 6$ & 4 & 3 & 2 & 2 & 2 \\
\hline
\rowcolor{Gray}
2 & 4 & 2 & $n \geq 6$ & 4 & 3 & 2 & 3 & 2 \\
\hline
\end{tabular}

\subsection{$[3,3,3]$ cusp}

\begin{tabular}{|c|c|c|c|c|c|c|c|c|}
\hline
$a_1$ &  $a_2$ &  $a_3$ &  $a_4$ &  $a_5$ &  $a_6$ &  $a_7$ &  $a_8$ &  $a_9$  \\
\hline
\hline
3 & 3 & 2 & 3 & 3 & 4 & 2 & 2 & 2  \\
\hline
3 & 3 & 2 & 3 & 3 & 4 & 2 & 2 & 3  \\
\hline
3 & 3 & 2 & 3 & 3 & 4 & 2 & 3 & 2  \\
\hline
3 & 3 & 2 & 3 & 3 & 4 & 3 & 2 & 2  \\
\hline
3 & 3 & 2 & 3 & 3 & 4 & 4 & 2 & 2  \\
\hline
3 & 3 & 2 & 3 & 3 & 4 & 5 & 2 & 2  \\
\hline
3 & 3 & 2 & 3 & 3 & 5 & 2 & 2 & 2  \\
\hline
3 & 3 & 2 & 3 & 3 & 5 & 2 & 2 & 3  \\
\hline
3 & 3 & 2 & 3 & 3 & 5 & 2 & 3 & 2  \\
\hline
3 & 3 & 2 & 3 & 3 & 5 & 3 & 2 & 2  \\
\hline
3 & 3 & 2 & 3 & 3 & 5 & 4 & 2 & 2  \\
\hline
\end{tabular}
\quad
\begin{tabular}{|c|c|c|c|c|c|c|c|c|}
\hline
$a_1$ &  $a_2$ &  $a_3$ &  $a_4$ &  $a_5$ &  $a_6$ &  $a_7$ &  $a_8$ &  $a_9$  \\
\hline
\hline
3 & 3 & 2 & 3 & 3 & 5 & 5 & 2 & 2  \\
\hline
3 & 3 & 2 & 4 & 3 & 4 & 2 & 2 & 2  \\
\hline
3 & 3 & 2 & 4 & 3 & 4 & 2 & 2 & 3  \\
\hline
3 & 3 & 2 & 4 & 3 & 4 & 2 & 3 & 2  \\
\hline
3 & 3 & 2 & 4 & 3 & 5 & 2 & 2 & 2  \\
\hline
3 & 3 & 2 & 4 & 3 & 5 & 2 & 2 & 3  \\
\hline
3 & 3 & 2 & 4 & 3 & 5 & 2 & 3 & 2  \\
\hline
3 & 3 & 2 & 4 & 3 & 5 & 3 & 2 & 2  \\
\hline
3 & 3 & 2 & 5 & 3 & 5 & 2 & 2 & 2  \\
\hline
3 & 3 & 2 & 5 & 3 & 5 & 2 & 2 & 3  \\
\hline
3 & 3 & 2 & 5 & 3 & 5 & 2 & 3 & 2  \\
\hline
\end{tabular}


\noindent Department of Mathematics \& Computer Science,\\
Eastern Illinois University,\\
600 Lincoln Avenue,\\
Charleston, IL 61920.\\
Email: gslakeland@eiu.edu

\noindent Department of Mathematics \& Computer Science,\\
Eastern Illinois University,\\
600 Lincoln Avenue,\\
Charleston, IL 61920.\\
Email: cgbarnett@eiu.edu

\end{document}